\documentclass[12pt]{amsart}
\usepackage{amsmath, amssymb, amsthm, amsfonts, mathrsfs}
\usepackage{cite}
\usepackage{amscd}
\usepackage{url}
\usepackage{graphicx}
\usepackage{caption}
\usepackage{subcaption}
\usepackage{fullpage}

\newtheorem{theorem}{Theorem}

\theoremstyle{definition}

\newcommand{\ess}[1]{\hat{#1}}
\newcommand{\fixation}{\rho}
\newcommand{\incentive}{\varphi}

\makeatletter
\g@addto@macro{\endabstract}{\@setabstract}
\newcommand{\authorfootnotes}{\renewcommand\thefootnote{\@fnsymbol\c@footnote}}%
\makeatother
\begin{document}
% \title{Incentive Processes in Finite Populations}
% \author{Marc Harper}
% \address{University of California Los Angeles}
% \email{marcharper@ucla.edu} 
% \author{Dashiell Fryer}
% \authorinfo{\texttt{c}}
\date{}
% \subjclass[2000]{Primary: 37N25; Secondary: 91A22, 94A15}
% \keywords{evolutionary game theory, information geometry, information divergence, replicator equation, Bayesian inference, information geometry, Fisher information}

\begin{center}
  \LARGE 
  Incentive Processes in Finite Populations \par \bigskip

  \normalsize
  \authorfootnotes
  Marc Harper\footnote{corresponding author, email: marcharper@ucla.edu}\textsuperscript{1}, Dashiell Fryer\textsuperscript{2}\par \bigskip

  \textsuperscript{1}Department of Genomics and Proteomics, UCLA, \par
  \textsuperscript{2}Department of Mathematics, Pomona College\par \bigskip
  
%   \today
\end{center}

\begin{abstract}
We define the incentive process, a natural generalization of the Moran process incorporating evolutionary updating mechanisms corresponding to well-known evolutionary dynamics, such as the logit, projection, and best-reply dynamics. Fixation probabilities and internal stable states are given for a variety of incentives, including new closed-forms, as well as results relating fixation probabilities for members of two one-parameter families of incentive processes. We show that the behaviors of the incentive process can deviate significantly from the analogous properties of deterministic evolutionary dynamics in some ways but are similar in others. For example, while the fixation probabilities change, their ratio remains constant.
\end{abstract}

% \maketitle

\section{Introduction}
In this manuscript we define an incentive-based generalization of the Moran process \cite{moran1958random} \cite{moran1962statistical} \cite{taylor2004evolutionary} \cite{nowak2006evolutionary} called the incentive process, which captures Fermi selection, \cite{traulsen2009stochastic}, and gives analogous Markov processes for every possible incentive. In \cite{fryer2012existence}, Fryer introduces a functional parameter called an incentive and defines the continuous incentive dynamics, simultaneously describing many common dynamics in evolutionary game theory \cite{hofbauer2003evolutionary} \cite{hofbauer1998evolutionary}, including the replicator, best-reply, logit, and projection dynamics. An incentive can be thought of as mediating the relationship between the population state and the fitness landscape. The authors explored the stability theory of the these dynamics and various generalizations in \cite{harper2012stability} \cite{fryer2012kullback} (see also \cite{harper2011escort}). In these studies, a simple relationship emerged for the replicator equation and its incentive as a means of bringing incentives to finite population dynamics. 

We investigate primarily two one-parameter families of incentives which we call the $q$-replicator and $q$-Fermi incentives. These families include four of the more common incentives in use: the projection, replicator, logit, and Fermi incentives. Recently human population growth in Spain has been shown to follow patterns of exponential growth with scale-factors $q \neq 1$ \cite{hernando2013workings}, and we consider the implications of such alternative values of $q$ in finite population dynamics. In essence, the incentive process uses \emph{incentive-proportionate} selection rather than simply $\emph{fitness-proportionate}$ selection as the Moran process (and replicator incentive) use as a basis of capturing natural selection. As we will see, the fitness-proportionate selection case has several special properties, but is by no means the only incentive suitable to model evolutionary processes. For example, the ratio of fixation probabilities satisfies the same relation as those of the Moran process, and so results regarding fitness flux for the Moran and Wright-Fisher processes apply to the $q$-replicator family as well \cite{sella2005application} \cite{mustonen2010fitness}.

The family of processes captured by incentives includes those studied by many authors, e.g. the Fermi process \cite{traulsen2009stochastic}. We note that this family captures a similar set of processes as Sandholm's microfoundations approach using revisions protocols \cite{sandholm2005excess} \cite{sandholm2010stochastic} \cite{sandholm2012stochastic}. The behavior of the continuous incentive dynamics on the interior of the probability simplex are qualitatively similar to the replicator equation, at least locally \cite{harper2012stability}; for the incentive process we find that the choice of incentive can exert a significant effect on the classical quantities associated to the Moran process, such as the fixation probabilities.

\section{Incentive Moran Process}

The incentive process was briefly introduced in the appendix of \cite{harper2013inherent}, where it was shown that the process has a particular bound on entropy rate. Let a population be composed of two types $A$ and $B$ of size $N$ with $i$ individuals of type $A$ and $N-i$ individuals of type $B$. The incentive process is a Markov process on the population states defined by the following transition probabilities, corresponding to a birth-death process where birth is incentive-proportionate and death is uniformly random:

\begin{align}\label{incentive_process}
T_{i \to i+1} &= \frac{\incentive_A(i)}{\incentive_A(i) + \incentive_B(i)} \frac{N-i}{N} \notag \\
T_{i \to i-1} &= \frac{\incentive_B(i)}{\incentive_A(i) + \incentive_B(i)} \frac{i}{N}, \\
T_{i \to i} &= 1 - T_{i \to i+1} - T_{i \to i-1}, \notag
\end{align}
The incentives $\incentive_A(i) = i f_A(i)$ and $\incentive_B(i) = (N-i) f_B(i)$ give the Moran process, where a traditional choice for the fitness landscape is:
\begin{align*}
f_A(i) &= \frac{a(i-1) + b(N-i)}{N-1} \\
f_B(i) &= \frac{ci + d(N-i-1)}{N-1}
\end{align*}
for a game matrix defined by
\[ A = \left( \begin{matrix}
 a & b\\
 c & d
\end{matrix} \right). \]

Table \ref{incentives_table} lists incentive functions for many common dynamics. Though we primarily investigate two one-parameter families of incentives defined in terms of a fitness landscape, we note that the incentive need not depend on the fitness landscape or the population state. The Fermi process of Traulsen et al is the $q$-Fermi for $q=1$ \cite{traulsen2009stochastic}, and for $q=0$ is called as the logit incentive. The $q$-replicator incentive has previously been studied in the context of evolutionary game theory \cite{harper2011escort} \cite{harper2012stability} and derives from statistical-thermodynamic and information-theoretic quantities \cite{tsallis1988possible}. Other than the best-reply incentive, which we will consider primarily as a limiting case of the $q$-Fermi, we will assume that all incentives are positive definite $\incentive(k) > 0$ if $0 < k < N$ to avoid the restatement of trivial hypotheses in the results that follow. For the $q$-Fermi incentives, positivity is of course guaranteed for all landscapes. This also implies that the fixation probability satisfies $0 < \fixation < 1$ for any such incentive.

\begin{figure}[h]
    \centering
    \begin{tabular}{|c|c|}
        \hline
        Dynamics & Incentive \\ \hline
        Projection & $\incentive_T(i) = f_T(i)$\\ \hline
        Replicator & $\incentive_T(i) = x_T f_T(i)$\\ \hline
        $q$-Replicator & $\incentive_T(i) = x_T^{q} f_T(i)$\\ \hline
        Logit & $\displaystyle{ \incentive_T(i) = \frac{\text{exp}(\beta f_T(x))}{\sum_{j \in \{A, B\}}{\text{exp}(\beta f_j(x))}}}$ \\ \hline
        Fermi & $\displaystyle{ \incentive_T(i) = \frac{x_T \text{exp}(\beta f_T(x))}{\sum_{j \in \{A, B\}}{x_j \text{exp}(\beta f_j(x))}}}$ \\ \hline
        $q$-Fermi & $\displaystyle{ \incentive_T(i) = \frac{x_T^{q}\text{exp}(\beta f_T(x))}{\sum_{j \in \{A, B\}}{x_T^{q}\text{exp}(\beta f_j(x))}}}$ \\ \hline
%         Best Reply & $\incentive_T(x) = BR_T(i) - \frac{x_T}{N}$ \\ \hline
        Best Reply & $\incentive_T(i) = x_T BR_T(i)$ \\ \hline
    \end{tabular}
    \caption{Incentives for some common dynamics, where type $T \in \{ A, B \}$, $x_A = i$, $x_B = N-i$, and $BR_T(i)$ is the best reply to state $x = (i, N-i)$ and is valued in $\{0,1\}$. The projection incentive \cite{sandholm2008projection} is simply the $q$-replicator with $q=0$. Another way of looking at the projection incentive on the Moran landscape is simply as a \emph{constant incentive}, since the incentive itself is a constant function in $i$. The logit incentive is the $q$-Fermi with $q=0$. For more examples see Table 1 in \cite{fryer2012existence}.}
    \label{incentives_table}
\end{figure}

\section{Fixation Probabilities}

For the Moran process the probability of fixation $\fixation_i$ of type A when the population starts in state $(i, N-i)$ is:
\begin{equation}
 \fixation_i = \frac{s_{0, i-1}}{s_{0, N-1}} \qquad \text{where} \qquad s_{i, j} = \sum_{k=i}^{j}{t_k} \qquad \text{and} \qquad t_j = \prod_{k=1}^{j}{\frac{f_B(k)}{f_A(k)}} 
\label{moran_fixation}
\end{equation}
\noindent where we use the notation to that of \cite{antal2006fixation} (see also \cite{taylor2004evolutionary}). We transform this for the incentive process simply by substituting $\incentive_A(i) = i f_A$ and $\incentive_B(i) = (N-i) f_B(i)$, giving
\begin{align} \label{incentive_fixation}
t_j = \prod_{k=1}^{j}{\frac{k\incentive_B(k)}{(N-k) \incentive_A(k)}} &= \left(\prod_{k=1}^{j}{\frac{\incentive_B(k)}{\incentive_A(k)}}\right) \left(\prod_{k=1}^{j}{\frac{k}{N-k}} \right) \notag \\
&= \left(\prod_{k=1}^{j}{\frac{\incentive_B(k)}{\incentive_A(k)}}\right) {\binom{N-1}{j}}^{-1}
\end{align}

Note that this formula can also be derived directly in terms of the transition probabilities (\ref{incentive_process}). It is well-known that the fixation probabilities for the Moran process are
\[ \rho_i = \frac{1 - r^{-i}}{1 - r^{-N}}, \] 
when the fitness landscape is given by the game matrix $a=b=r$ and $c=d=1$, which we refer to as the \emph{Moran landscape}. We are mostly concerned with the fixation probability when the starting state is $i=1$ and will denote it simply by $\fixation_A$ or simply $\fixation$ when unambiguous; similarly $\fixation_B$ will denote the probability of a single $B$-individual fixating, which is given by $\fixation_B = 1 - \fixation_{N-1} = t_{N-1}/s_{0, N-1}$. We then have that
\begin{equation} \label{fixation_ratio}
 \fixation_B / \fixation_A = t_{N-1} = \prod_{k=1}^{N-1}{\frac{k\incentive_B(k)}{(N-k) \incentive_A(k)}} 
\end{equation}
where for the Moran process we simply have
\[ \fixation_B / \fixation_A = \prod_{k=1}^{N-1}{\frac{f_B(k)}{f_A(k)}}, \]
as in \cite{taylor2004evolutionary} and \cite{nowak2006evolutionary}.

\subsection{Explicit Closed-Forms for Fixation Probabilities}

In \cite{harper2012stability} the authors consider an incentive called the $q$-replicator incentive. The projection incentive (corresponding to the projection dynamic \cite{sandholm2008projection}) is the special case where $q=0$ and the replicator incentive when $q=1$. Although Equation \ref{incentive_fixation} determines the fixation probabilities for this incentive, a pleasing closed form does not readily emerge for arbitrary $q$. Nevertheless the fixation probabilities can sometimes be expressed in terms of hypergeometric coefficients (though in some cases no closed-form can be given \cite{calkin1998factors}). For $q=2$, the closed form is
\[\rho_{q=2}(r) = \left(1 + \frac{1}{r} \right)^{1-N}\]
An interesting special case is that of the neutral landscape with $r=1$, though closed-form expressions are still hard to come by other than for $q=1,2,3$. The $q$-replicator incentive shares some of the intuitive properties of the classical case $q=1$. For instance, the fixation probabilities of each of the two types are equal for the neutral landscape for all $q$, which can easily be deduced from Equation \ref{fixation_ratio}. Many properties, however, are dependent on the value of the parameter. For $q=0$, a straightforward calculation shows that
\[ \fixation_{q=0}(N) = \frac{1}{\sum_{k=1}^{N-1}{\binom{N-1}{k}^{-1}}} = \frac{2^N}{N} \frac{1}{\sum_{k=1}^{N}{\frac{2^k}{k}}}, \]
where the last equality uses an identity from \cite{sprugnoli2006sums}. From this formula we can deduce that $1/N < \fixation_{q= 0}(N) < 1/2$ for $N>2$. Perhaps surprisingly, $\lim_{N\to \infty}{\fixation_1(N)} \to 0$ only holds for $q \geq 1$; for $q=0$, the limit is $1/2$. The $2$-replicator incentive has fixation probability $\fixation_{q=2} = 2^{1-N}$ and the $3$-replicator incentive has fixation probability $\fixation_{q=3} = \binom{2N-2}{N-1}^{-1}$. Compare to $\fixation_{q=1} = 1/N$ for the $1$-replicator incentive. We have that $\fixation_{q=3}(N) < 2^{1-N} = \fixation_{q=2}(N) < 1/N = \fixation_{q=1}(N) < \fixation_{q=0}(N)$ for $N > 2$ (they are all equal to $1/2$ when $N=2$). See Figure \ref{figure_neutral_fixations}. %Since the neutral fixation probabilities deviate from $1/N$ for $q \neq 1$, one interpretation of the parameter $q$ is that it affects the intensity of selection, though not in as simple a manner as $\beta$.

% \begin{figure}
%     \centering
%     \includegraphics[scale=0.2]{neutral_q_rep_fixations.png}
%     \caption{Fixation probabilities for the neutral fitness landscape for the $q$-replicator incentive as a function of population size. Top to bottom: $q=0, 0.5, 0.8, 0.9, 1, 2$.}
%     \label{figure_neutral_fixations}
% \end{figure}
% 
% \begin{figure}
%     \centering
%     \includegraphics[scale=0.2]{neutral_heatmap.png}
%     \caption{Heatmap of fixation probabilities for the neutral fitness landscape for the $q$-replicator incentive as a function of population size. The parameter $q$ is on the vertical axis; population size $N$ on the horizontal. Compare to Figure \ref{figure_neutral_fixations}.}
%     \label{figure_neutral_fixations_heatmap}
% \end{figure}

\begin{figure}[h]
        \begin{subfigure}[b]{0.5\textwidth}
            \centering
            \includegraphics[width=\textwidth]{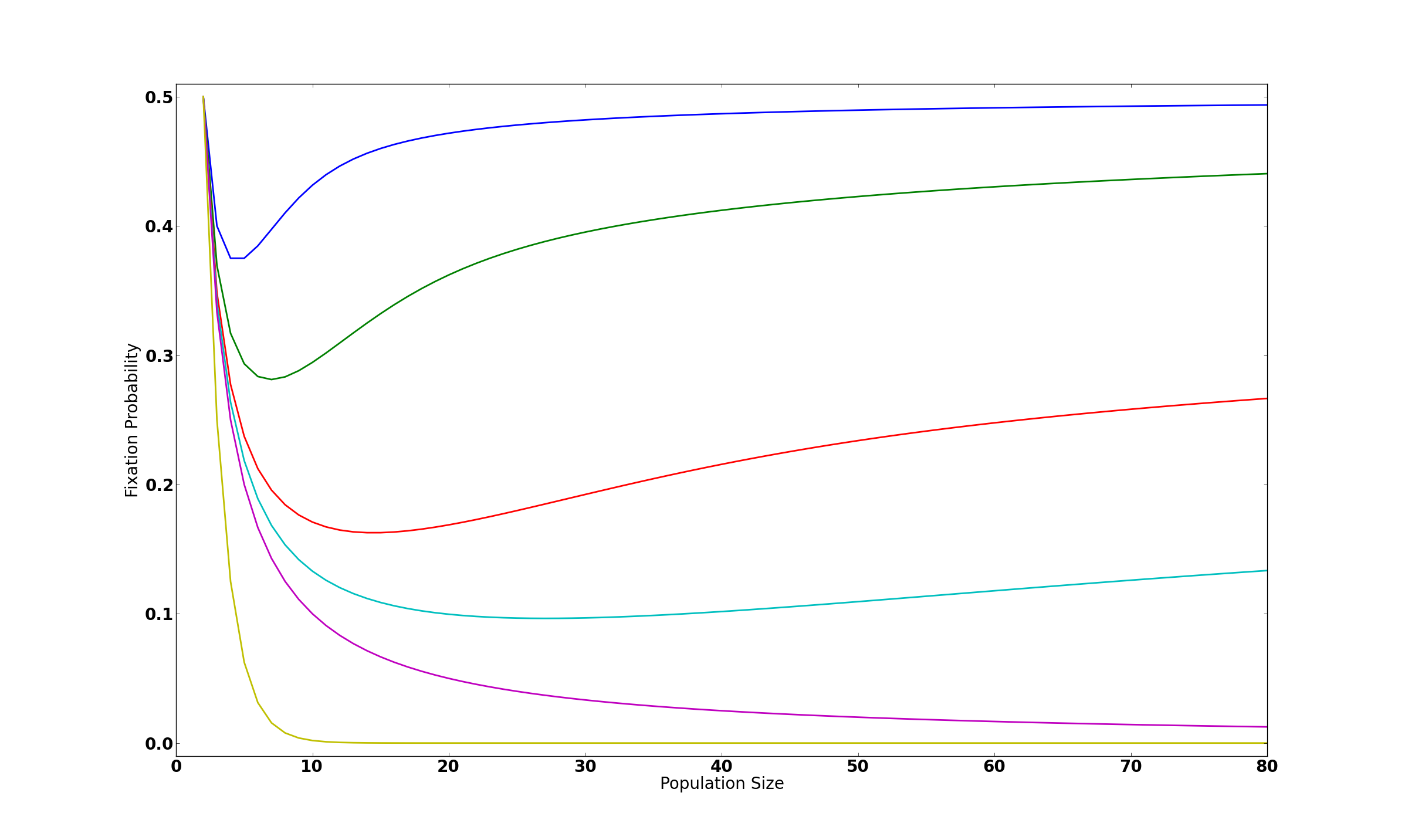}
        \end{subfigure}%
        ~ %add desired spacing between images, e. g. ~, \quad, \qquad etc. 
        \begin{subfigure}[b]{0.5\textwidth}
            \centering
            \includegraphics[width=\textwidth]{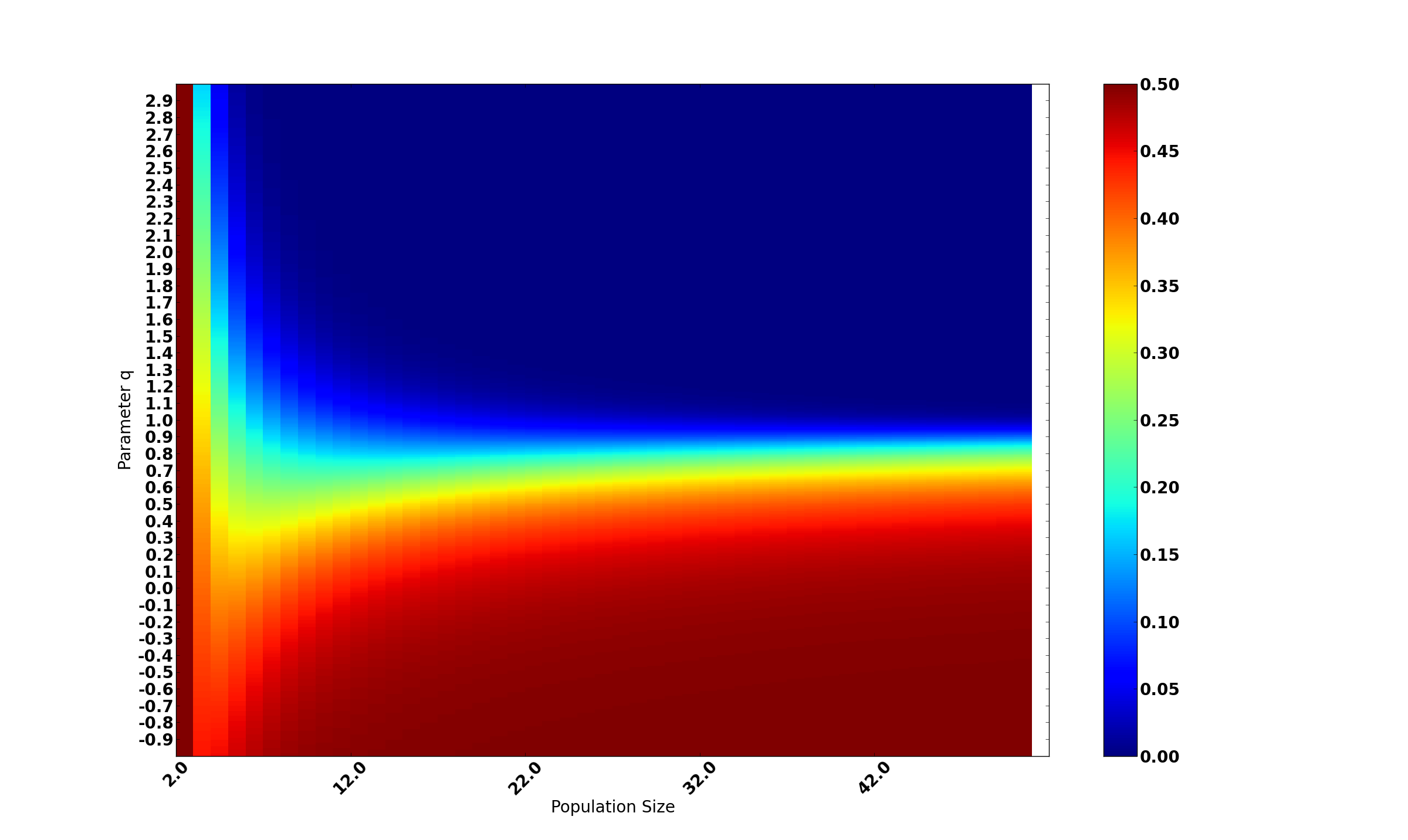}
        \end{subfigure}\\
        \caption{Left: Fixation probabilities for the neutral fitness landscape for the $q$-replicator incentive as a function of population size. Top to bottom: $q=0, 0.5, 0.8, 0.9, 1, 2$. Right: Heatmap of fixation probabilities for the neutral fitness landscape for the $q$-replicator incentive as a function of population size. The parameter $q$ is on the vertical axis; population size $N$ on the horizontal.}
        \label{figure_neutral_fixations}
\end{figure} 

Any closed form for the $q$-replicator gives a closed form for the $q$-Fermi for the same $q$ with the substitution $r \mapsto e^{\beta (r-1)}$, so the discussion above also gives a closed-form for the $2$-Fermi. Note also that the fixation probability for the best-reply dynamic is a step function jumping from zero to one at $r=1$, and is obtainable as a limit of the $q$-Fermi as $\beta \to \infty$.

\subsection{Monotonicity of Fixation Probabilities}

We have seen a particular case of the departure of behavior between deterministically defined incentive dynamics and the incentive process. The continuous incentive dynamic behaves qualitatively similarly to the replicator dynamic with respect to fixation for the Moran landscape. That is, if $r>1$, fixation is guaranteed for \emph{any incentive} listed in Table \ref{incentives_table} when the fitness landscape is given by the Moran landscape with matrix $A_r$. For the incentive process, the fixation probabilities are dependent on the incentive and its parameters, so the \emph{long-run} behaviors are qualitatively different. More precisely, we have the following.

\begin{theorem}
Let $\incentive$ be an incentive, let $N > 2$, and define a new incentive $\psi_{A,q}(i) = i^q \incentive_A(i)$ and $\psi_{B,q}(i) = (N-i)^q \incentive_B(i)$. Then for $p < q$, $\fixation_{A, p} > \fixation_{A, q}$.
\label{theorem_fixation_decreasing}
\end{theorem}
\begin{proof}
First note that
\[ \prod_{k=1}^{j}{\frac{N-k}{k}} = \binom{N-1}{j} > 1\] if $j < N-1$ and $N>2$. Then
\begin{align*}
\prod_{k=1}^{j}{\frac{(N-k)^{p-1}\incentive_B(k)}{k^{p-1} \incentive_A(k)}} &= \prod_{k=1}^{j}{\frac{\incentive_B(k)}{\incentive_A(k)}} \left(\prod_{k=1}^{j}{\frac{N-k}{k}}\right)^{p-1} \\
&< \prod_{k=1}^{j}{\frac{\incentive_B(k)}{\incentive_A(k)}} \left(\prod_{k=1}^{j}{\frac{N-k}{k}}\right)^{q-1} \\
&= \prod_{k=1}^{j}{\frac{(N-k)^{q-1}\incentive_B(k)}{k^{q-1} \incentive_A(k)}}
\end{align*}
The inequality is then reversed in the equation for the fixation probability (\ref{incentive_fixation}).
\end{proof}

There are immediate consequences of Theorem \ref{theorem_fixation_decreasing}. In particular, the fixation probability of type $A$ for the projection process is greater than that of the replicator incentive, which is greater than that of the $2$-replicator incentive, and so on (including the special case discussed before the theorem). For the Moran landscape, since there are known closed-forms for $q=1$ and $q=2$, there are explicit upper and/or lower bounds for various $q<1$, $1<q<2$, and $q>2$. The same holds for the $q$-Fermi incentive.

A somewhat similar result holds for variation in $\beta$ for the $q$-Fermi. Clearly Theorem \ref{theorem_fixation_decreasing} applies to the $q$-Fermi process for fixed $\beta$. For variable $\beta$, fixed $q$, and the Moran landscape, the analogous theorem for variation in $\beta$ depends on the value of $r$. Since the effective relative fitness is $e^{\beta (r-1)}$, the fixation probability increases or decreases in $\beta$ depending on whether $r-1 > 0$ or $r-1 < 0$. See Figure \ref{figure_q_logit_fixations}. %This discussion also shows that a theorem as simple as Theorem \ref{theorem_fixation_decreasing} cannot hold for variation in $\beta$.

\begin{figure}
    \centering
    \includegraphics[width=0.6\textwidth]{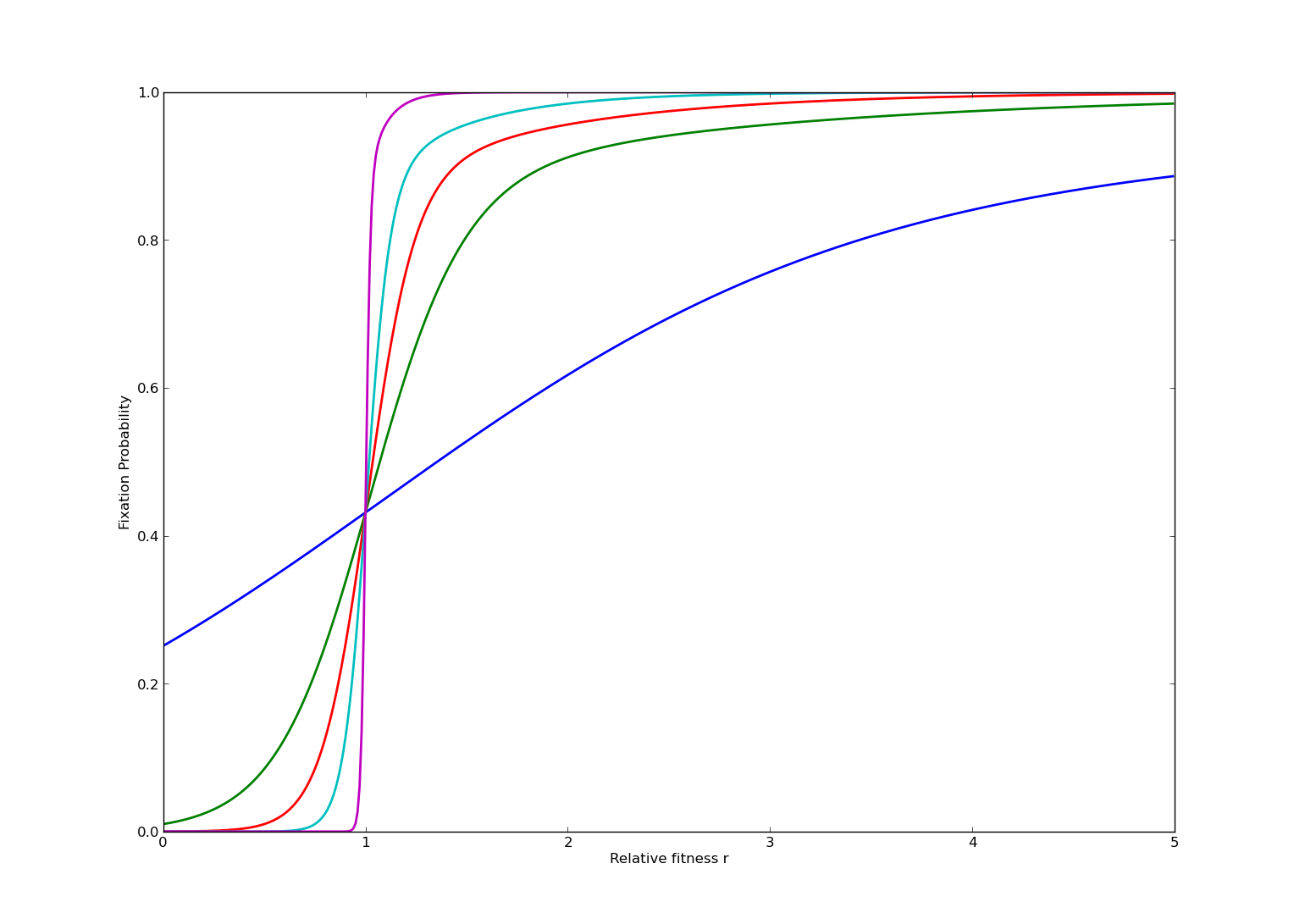}
    \caption{Fixation Probabilities for the Moran landscape for the Logit incentive versus relative fitness $r$, population size $N=10$. From the intersection point outward: $\beta=0.1, 0.5, 1, 2, 10$. The fixations approach that of the best-reply incentive as $\beta \to \infty$ (a step function).}
    \label{figure_q_logit_fixations}
\end{figure}

\subsection{Ratios of Fixation Probabilities}

Our next theorem shows that some properties of the fixation probabilities are independent of $q$. 
% \begin{theorem}\label{theorem_fixation_ratios}
% For the $q$-replicator and $q$-Fermi incentives, the ratio $\frac{\fixation_B}{\fixation_A}$ is constant in $q$.
% \end{theorem}
% \begin{proof}
% \begin{align*}
% \frac{\fixation_B}{\fixation_A} = t_{N-1} &= \left(\prod_{k=1}^{N-1}{\frac{\incentive_B(k)}{\incentive_A(k)}}\right) {\binom{N-1}{N-1}}^{-1}\\
% &= \left(\prod_{k=1}^{N-1}{\frac{f_B(k)}{f_A(k)}}\right) {\binom{N-1}{N-1}}^{q-1}\\
% &= \left(\prod_{k=1}^{N-1}{\frac{f_B(k)}{f_A(k)}}\right)
% \end{align*}
% The proof for the $q$-Fermi is analogous.
% \end{proof}

\begin{theorem}\label{theorem_fixation_ratios}
Let $\incentive$ be an incentive and define a new incentive $\psi_{A,q}(i) = i^q \incentive_A(i)$ and $\psi_{B,q}(i) = (N-i)^q \incentive_B(i)$. Then the ratio of fixation probabilities $\fixation_B / \fixation_A$ is independent of $q$.
\end{theorem}
\begin{proof}
From equations \ref{incentive_fixation} and \ref{fixation_ratio},
\[
 \left(\prod_{k=1}^{N-1}{\frac{\incentive_B(k)}{\incentive_A(k)}}\right) {\binom{N-1}{N-1}}^{-1} = t_{N-1} = \left(\prod_{k=1}^{N-1}{\frac{\incentive_B(k)}{\incentive_A(k)}}\right) {\binom{N-1}{N-1}}^{q-1}
\]
So the ratio of fixation probabilities is the same for $\psi$ and $\incentive$.
\end{proof}

For the $q$-replicator and $q$-Fermi incentives, the ratio of fixation probabilities is constant in $q$, even though the fixation probabilities themselves are not the same (as we have already seen for the neutral landscape, and as implied by Theorem \ref{theorem_fixation_decreasing}). Moreover the ratio of fixation probabilities for these two families of incentives are the same as for the Moran process!

Together the two theorems of this section show that even though the fixation probabilities vary with $q$, their ratio does not; Theorem \ref{theorem_fixation_ratios} extends Theorem \ref{theorem_fixation_decreasing} to include $\fixation_B$ (though a direct calculation could have done the same). Theorem \ref{theorem_fixation_ratios} allows us to use eariler calculations for $\fixation_A$ to easily calculate $\fixation_B$, e.g., for $2$-replicator, that $\fixation_B = (1+r)^{1-N}$, since $\fixation_B/\fixation_A = r^{1-N}$ for the Moran process \cite{nowak2006evolutionary}. (There is a $r \mapsto 1/r$ symmetry in fixation probabilities, as for the Moran process.) For the $q$-Fermi and the Moran landscape, the ratio of fixation probabilities is
\begin{equation} \label{fermi_ratio}
\log \frac{\fixation_A}{\fixation_B} = -\sum_{k=1}^{N-1}{\beta \left( f_B(k) - f_A(k) \right)} = \beta (r-1) (N-1),
\end{equation}
giving a family-wide relation between the population size, the strength of selection, and the relative fitness. From this we can see that for large $\beta$, fixation favors the better reply (with sign dictated by $r-1$), as expected. If $r>1$, larger populations favor fixation of type $A$, since the probability that $A$ is randomly chosen for replacement for small populations is diminished; the strength of selection parameter $\beta$ affects the relative fixation probabilities in the same manner, amplifying the difference in relative fitness, and possibly tempering the effects of the population size (small or large) depending on the values of $N$ and $\beta$ (Equation \ref{fermi_ratio}). 

We also have from Theorem \ref{theorem_fixation_ratios} that for large population sizes the fixation probabilities of the two types converge to zero and one, for all $q$ for the $q$-replicator, for $r\neq 1$. This is simply because $\frac{\fixation_B}{\fixation_A} = r^{1-N}$ (see \cite{nowak2006evolutionary} for $q=1$), and since this ratio is constant for all $q$. For a large population, either $\frac{\fixation_B}{\fixation_A}$ or its reciprocal approaches 0, depending on if $r>1$ or $r<1$. This is not true for the neutral landscape $r=1$, where the two fixation probabilities are equal for all $N$. Hence despite the affect the parameter $q$ can have on fixation probabilities, some properties of the Moran process are unchanged. Note though that for the $q$-Fermi incentive this is also true -- one of the two fixation probabilities goes to zero as either $N$ or $\beta$ gets large, or they are equal if $r=1$.

Theorem \ref{theorem_fixation_ratios} also implies that the boundary elements of the stationary probabilities of the $q$-replicator incentive process are independent of $q$, as mutation rates go to zero. This is because the stationary probability $\pi(0)$ of state $(0,N)$ is $\fixation_A / (\fixation_A + \fixation_B) = 1 / (1 + \fixation_B/\fixation_A)$, which is independent of $q$ by the theorem. Indeed in this case the stationary distributions are identical, but direct calculations show that for nonzero mutation rates the stationary distributions are not equal. There are further implications for the entropy rate of the processes and other quantities related to the stationary distribution, but is beyond the scope of this investigation. (See \cite{fudenberg2004stochastic} and \cite{harper2013inherent}.)

\subsection{Fitness Flux and Fisher's Fundamental Theorem}

There is another important consequence of Theorem \ref{theorem_fixation_ratios}. In \cite{mustonen2010fitness}, Mustonen and L\"{a}ssig introduce the \emph{fitness flux} theorem as an application of Crook's fluctuation theorem \cite{crooks1999entropy}, and show that there is a version of Fisher's Fundamental theorem of Natural Selection \cite{fisher1999genetical} for the Moran process and the Wright-Fisher process \cite{imhof2006evolutionary}. In our context, this amounts to essentially ``compressing'' the processes to a new process with just two states -- All Type A and All Type B -- while introducing a mutation from the fixated states to the states where the population has single A or B mutant (i.e. in the manner of \cite{claussen2005non}). More precisely, we set $T_{0 \to 1} = \mu = T_{N \to N-1}$, and set $T_{1 \to N} = \fixation_A$ and $T_{N-1 \to 0} = \fixation_B$. Then we further reduce to a two state process where $T_{\text{All A} \to \text{All B}} = \mu \fixation_A$ and $T_{\text{All B} \to \text{All A}} = \mu \fixation_B$. Then the analogous ratio to Equation \ref{fixation_ratio} is simply $\fixation_B / \fixation_A$, and is equal to $r^{1-N}$ for the Moran process.

Since Theorem \ref{theorem_fixation_ratios} shows that the ratio of fixation probabilities for the $q$-replicator incentive process are the same as those of the Moran process, the above discussion and derivations in \cite{mustonen2010fitness} also apply to the incentive process. So these incentives carry a natural version of Fisher's fundamental theorem, further justifying their use as models of selective processes. For a readable summary of the associated quantities and the relationship to the second law of thermodynamics, see \cite{smerlak}.

We note that this is remarkably similar to the state of affairs for deterministic dynamics in which the relative entropy yields a Lyapunov function and there is also an associated version of Fisher's Fundamental theorem \cite{harper2012stability} \cite{harper2011escort}.

\section{Incentive Stable States}

The incentive dynamic was defined in part to investigate the stability of incentive dynamics. In this section we leave the discussion of fixation probabilities and consider the appropriate notions of evolutionary equilibria. The natural generalization of an evolutionarily stable state to a process driven by an incentive is an \emph{incentive stable state} (ISS) \cite{fryer2012existence}. For the continuous incentive dynamic on the simplex, the ISS condition is simply \[\sum_i{\frac{\ess{x}_i - x_i}{x_i} \incentive_i(x)} > 0,\]
which gives the familiar evolutionarily stable state condition (ESS) $\hat{x} \cdot f(x) > x \cdot f(x)$ for the replicator incentive $\incentive_i(x) = x_i f_i(x)$. Fryer showed that the ISS condition for the best reply dynamic is again ESS, and in \cite{harper2012stability}, for 2x2 linear games, it was shown that an ISS for the projection dynamic is again an ESS, that is satisfies the equation $x_1 = (d-b) / (a-c + d-b)$, which is expected from the known Lyapunov theorems \cite{harper2012stability} \cite{sandholm2008projection}. This is not always the case, however -- an ISS can fail to be an ESS.

An analog of ESS exists for the Moran process \cite{nowak2006evolutionary} by setting equal the transition probabilities and solving for the state $i$, which yields
\begin{equation}\label{ess_finite}
 \frac{i}{N} = \frac{d-b + \frac{a-d}{N}}{a-c+d-b} 
\end{equation}
Note that this definition lacks a stability criterion, and often additional criteria are imposed, such as $\text{ESS}_N$ \cite{nowak2006evolutionary}. For the incentive process, we can find the ISS analogously by solving $(N-i) \incentive_A(i) = i \incentive_B(i)$, which we call \emph{ISS candidates} to indicate that these states have not yet been identified as stable. It is easy to see that this relation can lead to more interesting interior dynamics, since there can be multiple interior ISS candidates while there can be at most one ESS for the Moran process with linear fitness determined by a game matrix. Let us consider the $q$-replicator incentive for $q=0,1,2$. The case $q=1$ is Equation \ref{ess_finite} above. For $q=0$ and $q=2$, the ISS candidate condition gives curiously similar quadratic equations in $i/N$:
\begin{align*}
\left(a+c-b-d\right) \left(\frac{i}{N}\right)^2 &+ \left(d+2b-a-\frac{a+d}{N}\right)\frac{i}{N} + \left(\frac{a}{N} - b\right) = 0 \\
\left(a+c-b-d\right) \left(\frac{i}{N}\right)^2 &+ \left(b+2d-c -\frac{a+d}{N}\right)\frac{i}{N} + \left(\frac{d}{N} - d\right ) = 0,
\end{align*}
which of course in general need not have the same solutions for particular $a,b,c,$ and $d$. See Figure \ref{figure_q_rep_iss}. It is easy to see that if $N$ is even and $f_A(N/2) = f_B(N/2)$ then this ISS candidate equation is satisfied for all $q$ at $i = N/2$. For instance, a Hawk-Dove game matrix $a=1=d$, $b=2=c$, which has an internal ESS for the replicator dynamic at $(N/2, N/2)$, is an ISS candidate for any $q$. 

\begin{figure}
    \centering
    \includegraphics[width=0.65 \textwidth]{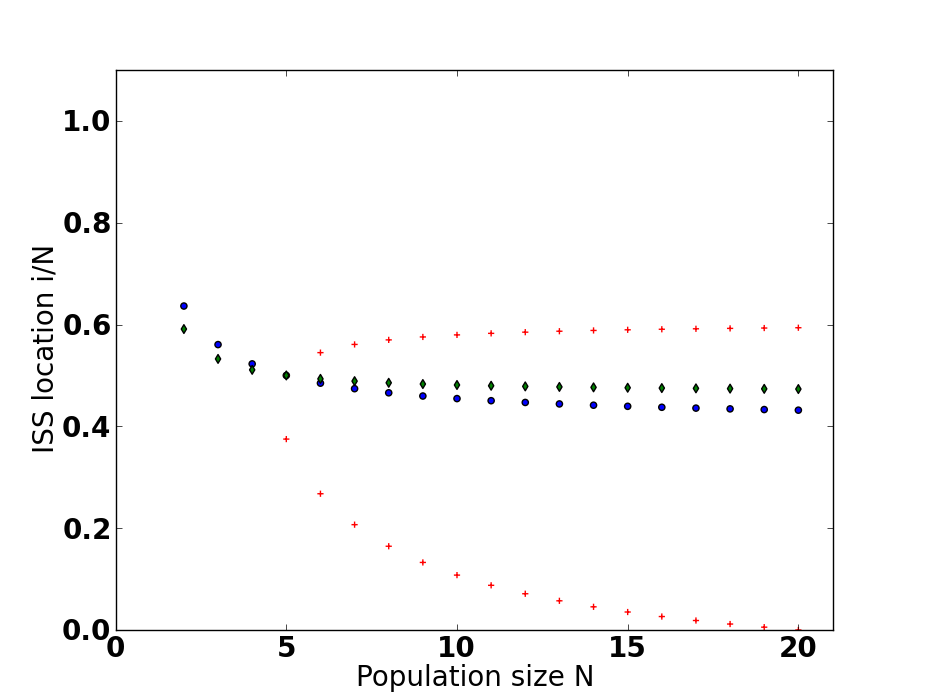}
    \caption{ISS candidates $i/N$ for $q$-replicator incentive. Blue dots: $q=1$, red $+$: $q=0$, green diamonds: $q=2$. Notice that for $q=0$, depending on $N$, there may be no ISS or two ISS candidates in contrast to the Moran process ($q=1$) for which there is a single interior ISS. The game matrix for all three incentives is given by $a=20,b=1,c=7,d=10$.}
    \label{figure_q_rep_iss}
\end{figure}

For the Moran process and landscape, there is no interior ISS candidate (no ESS) for the replicator incentive if $r \neq 1$; there is such an interior point for the $q$-replicator for all other values of $q$, including the projection incentive. This point is given by: 
\begin{equation}\label{r_iss}
\frac{i}{N} = \frac{1}{1+r^{\frac{1}{q-1}}} < 1. 
\end{equation}
Similarly for the $q$-Fermi. See Figure \ref{figure_r_iss}. For continuous dynamics, these points are stable for $q < 1$; for $q > 1$, the stability depends on the starting state of the dynamic \cite{harper2012stability}.

\begin{figure}[h]
        \begin{subfigure}[b]{0.5\textwidth}
            \centering
            \includegraphics[width=\textwidth]{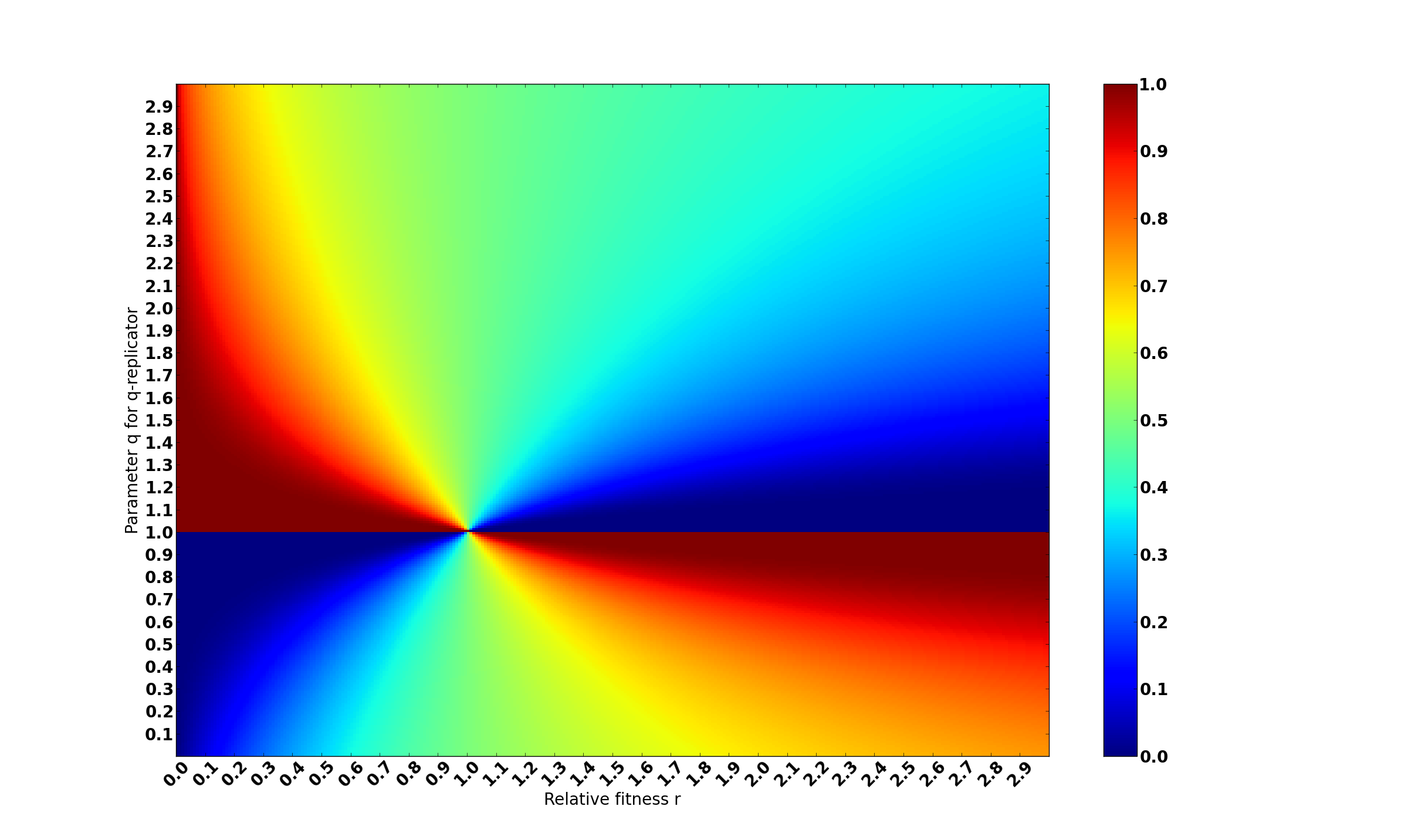}
        \end{subfigure}%
        ~ %add desired spacing between images, e. g. ~, \quad, \qquad etc. 
        \begin{subfigure}[b]{0.5\textwidth}
            \centering
            \includegraphics[width=\textwidth]{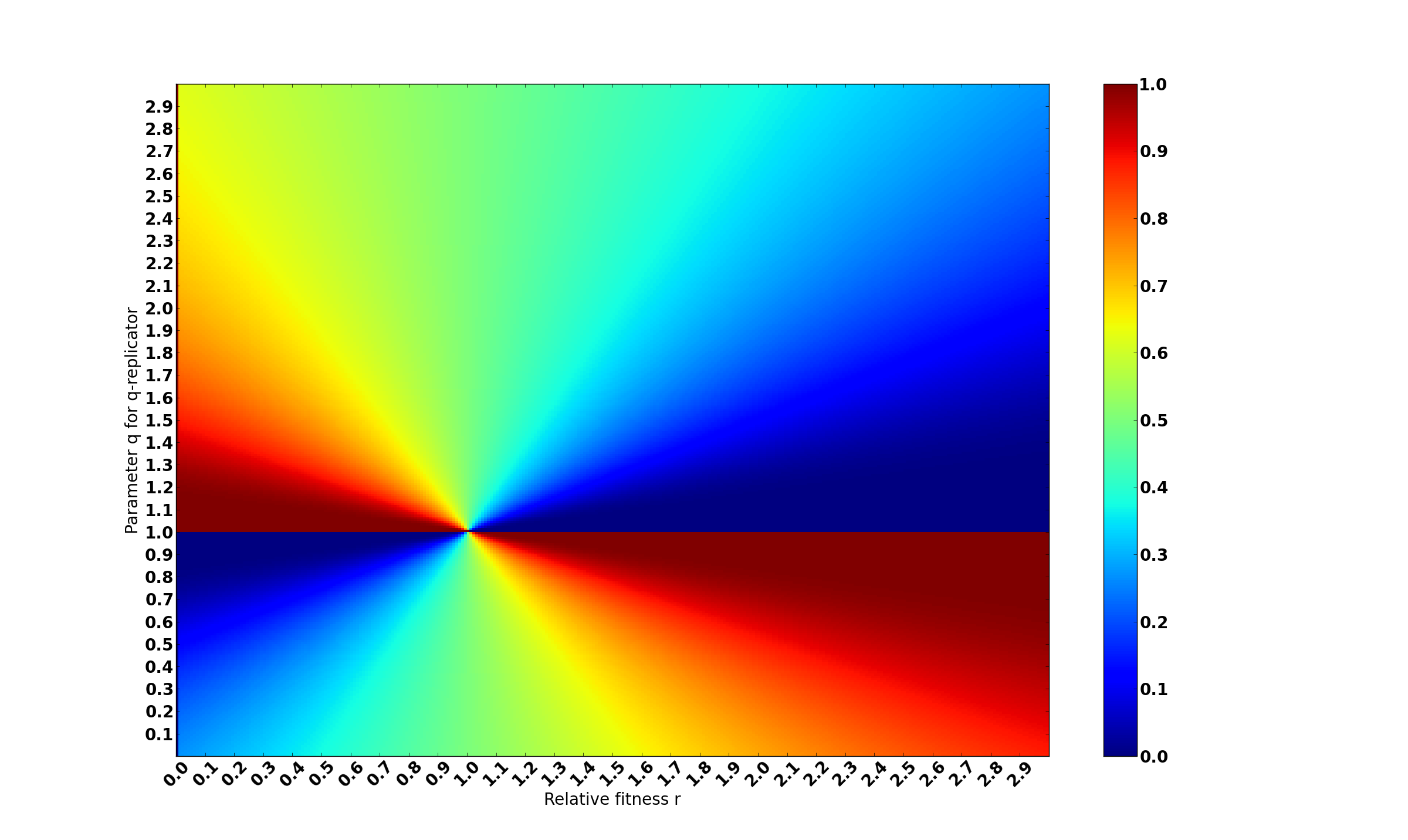}
        \end{subfigure}\\
        \caption{Left: ISS candidates $i/N$ for $q$-replicator incentive (Equation \ref{r_iss}). Right: Same for $q$-Fermi incentive with $\beta=1$. Notice that the lines of constancy are straightened by the exponential.}
    \label{figure_r_iss}
\end{figure} 

\section{Discussion}

We have seen that the choice of incentive can have a striking effect on the incentive process, even for ostensibly very similar incentives. The choice of incentive can dramatically alter the large population behavior of the process, notably for the fixation probabilities of the neutral fitness landscape. Fixation probabilities shift with changes in the parameter $q$ for the $q$-replicator and $q$-Fermi, nevertheless the ratio of the fixation probabilities of the two population types remains constant as $q$ varies, which has several consequences for the behavior of the process. Moreover, we see that the choice of incentive can internally shift the ``equilibrium'' of the process away from the boundary.

In some sense, we see ways in which fitness-proportionate selection as in the Moran process is special. In particular, it is the only incentive in the $q$-replicator family with neutral fixation probability $\fixation = 1/N$. The particular value $q=1$ is a qualitative behavior changing point for the internal equilibrium (see Figure \ref{figure_r_iss}), and is known to have unique properties related to dynamics on the probability simplex \cite{harper2011escort} \cite{c1982statistical}. This is in contrast to the analogous deterministic dynamics, where the behaviors of the various incentives are often much more similar, such as the fact that the internal equilibria of the projection and replicator dynamics are the same.

% Finally we note that the introduction of parameters into evolutionary dynamics is clearly demystifying which properties typical of evolutionary dynamics are due to definition, and which are more sensitive to slight changes in these parameters. This goes beyond simply unifying known dynamics into common model families. With the growing body of work in modeling in population biology (including human population growth) and other fields, such as statistical physics, we encourage others to consider similar alterations to evolutionary models, for the purposes of both modeling and mathematical analysis.

\subsection*{Methods}
All computations were performed with python code available at \url{https://github.com/marcharper/fixation}. All plots created with \emph{matplotlib} \cite{Hunter:2007}.

\subsection*{Acknowledgments}
This research was partially supported by the Office of Science (BER),
U. S. Department of Energy, Cooperative Agreement No. DE-FC02-02ER63421.

% \begin{itemize}
%  \item interesting entropy rate scenarios
%  \item Fitness Inference -- FPS distributions and birth-death generalizations
%  \item stochastic SDE for the limit as $N \to \infty$ (e.g. like Traulsen)
% \end{itemize}

\bibliography{ref}

\begin{thebibliography}{10}

\bibitem{antal2006fixation}
Tibor Antal and Istvan Scheuring.
\newblock Fixation of strategies for an evolutionary game in finite
  populations.
\newblock {\em Bulletin of mathematical biology}, 68(8):1923--1944, 2006.

\bibitem{calkin1998factors}
Neil~J Calkin.
\newblock Factors of sums of powers of binomial coefficients.
\newblock {\em ACTA ARITHMETICA-WARSZAWA-}, 86:17--26, 1998.

\bibitem{c1982statistical}
Nikola{\u\i} Chentsov.
\newblock {\em Statistical decision rules and optimal inference}.

\bibitem{claussen2005non}
J.C. Claussen and A.~Traulsen.
\newblock Non-gaussian fluctuations arising from finite populations: Exact
  results for the evolutionary moran process.
\newblock {\em Physical Review E}, 71(2):025101, 2005.

\bibitem{crooks1999entropy}
Gavin~E Crooks.
\newblock Entropy production fluctuation theorem and the nonequilibrium work
  relation for free energy differences.
\newblock {\em Physical Review E}, 60(3):2721, 1999.

\bibitem{fisher1999genetical}
Ronald~Aylmer Fisher.
\newblock {\em The genetical theory of natural selection: a complete variorum
  edition}.
\newblock OUP Oxford, 1999.

\bibitem{fryer2012existence}
Dashiell~EA Fryer.
\newblock On the existence of general equilibrium in finite games and general
  game dynamics.
\newblock {\em arXiv preprint arXiv:1201.2384}, 2012.

\bibitem{fryer2012kullback}
D.E.A. Fryer.
\newblock The kullback-liebler divergence as a lyapunov function for incentive
  based game dynamics.
\newblock {\em arXiv preprint arXiv:1207.0036}, 2012.

\bibitem{fudenberg2004stochastic}
Drew Fudenberg, Lorens Imhof, Martin~A Nowak, and Christine Taylor.
\newblock Stochastic evolution as a generalized moran process.
\newblock {\em Unpublished manuscript}, 2004.

\bibitem{harper2011escort}
Marc Harper.
\newblock Escort evolutionary game theory.
\newblock {\em Physica D: Nonlinear Phenomena}, 240(18):1411--1415, 2011.

\bibitem{harper2013inherent}
Marc Harper.
\newblock The inherent randomness of evolving populations.
\newblock {\em arXiv preprint arXiv:1303.1890}, 2013.

\bibitem{harper2012stability}
Marc Harper and Dashiell~EA Fryer.
\newblock Stability of evolutionary dynamics on time scales.
\newblock {\em arXiv preprint arXiv:1210.5539}, 2012.

\bibitem{hernando2013workings}
A~Hernando, R~Hernando, A~Plastino, and AR~Plastino.
\newblock The workings of the maximum entropy principle in collective human
  behaviour.
\newblock {\em Journal of The Royal Society Interface}, 10(78), 2013.

\bibitem{hofbauer1998evolutionary}
Josef Hofbauer and Karl Sigmund.
\newblock {\em Evolutionary games and population dynamics}.
\newblock Cambridge University Press, 1998.

\bibitem{hofbauer2003evolutionary}
Josef Hofbauer and Karl Sigmund.
\newblock Evolutionary game dynamics.
\newblock {\em Bulletin of the American Mathematical Society}, 40(4):479, 2003.

\bibitem{Hunter:2007}
J.~D. Hunter.
\newblock Matplotlib: A 2d graphics environment.
\newblock {\em Computing In Science \& Engineering}, 9(3):90--95, 2007.

\bibitem{imhof2006evolutionary}
L.A. Imhof and M.A. Nowak.
\newblock Evolutionary game dynamics in a wright-fisher process.
\newblock {\em Journal of mathematical biology}, 52(5):667--681, 2006.

\bibitem{moran1958random}
Patrick Alfred~Pierce Moran.
\newblock Random processes in genetics.
\newblock In {\em Mathematical Proceedings of the Cambridge Philosophical
  Society}, volume~54, pages 60--71. Cambridge Univ Press, 1958.

\bibitem{moran1962statistical}
Patrick Alfred~Pierce Moran et~al.
\newblock The statistical processes of evolutionary theory.
\newblock {\em The statistical processes of evolutionary theory.}, 1962.

\bibitem{mustonen2010fitness}
Ville Mustonen and Michael L{\"a}ssig.
\newblock Fitness flux and ubiquity of adaptive evolution.
\newblock {\em Proceedings of the National Academy of Sciences},
  107(9):4248--4253, 2010.

\bibitem{nowak2006evolutionary}
Martin~A Nowak.
\newblock {\em Evolutionary dynamics: exploring the equations of life}.
\newblock Harvard University Press, 2006.

\bibitem{sandholm2005excess}
William~H Sandholm.
\newblock Excess payoff dynamics and other well-behaved evolutionary dynamics.
\newblock {\em Journal of Economic Theory}, 124(2):149--170, 2005.

\bibitem{sandholm2010stochastic}
William~H Sandholm.
\newblock Stochastic evolutionary game dynamics: Foundations, deterministic
  approximation, and equilibrium selection.
\newblock {\em American Mathematical Society}, 201(1):1--1, 2010.

\bibitem{sandholm2012stochastic}
William~H Sandholm.
\newblock Stochastic imitative game dynamics with committed agents.
\newblock {\em Journal of Economic Theory}, 147(5):2056--2071, 2012.

\bibitem{sandholm2008projection}
William~H Sandholm, Emin Dokumac{\i}, and Ratul Lahkar.
\newblock The projection dynamic and the replicator dynamic.
\newblock {\em Games and Economic Behavior}, 64(2):666--683, 2008.

\bibitem{sella2005application}
Guy Sella and Aaron~E Hirsh.
\newblock The application of statistical physics to evolutionary biology.
\newblock {\em Proceedings of the National Academy of Sciences of the United
  States of America}, 102(27):9541--9546, 2005.

\bibitem{smerlak}
Matteo Smerlak.
\newblock The mathematical origin of irreversibility.
\newblock
  \url{http://johncarlosbaez.wordpress.com/2012/10/08/the-mathematical-origin-of-irreversibility/},
  2012.

\bibitem{sprugnoli2006sums}
Renzo Sprugnoli.
\newblock Sums of the reciprocals of central binomial coefficients.
\newblock {\em Integers}, 6:A27, 2006.

\bibitem{taylor2004evolutionary}
Christine Taylor, Drew Fudenberg, Akira Sasaki, and Martin~A Nowak.
\newblock Evolutionary game dynamics in finite populations.
\newblock {\em Bulletin of mathematical biology}, 66(6):1621--1644, 2004.

\bibitem{traulsen2009stochastic}
Arne Traulsen and Christoph Hauert.
\newblock Stochastic evolutionary game dynamics.
\newblock {\em Reviews of nonlinear dynamics and complexity}, 2:25--61, 2009.

\bibitem{tsallis1988possible}
Constantino Tsallis.
\newblock Possible generalization of boltzmann-gibbs statistics.
\newblock {\em Journal of statistical physics}, 52(1):479--487, 1988.

\end{thebibliography}
\bibliographystyle{plain}

\end{document}